\documentclass[12pt,a4paper,reqno]{amsart}
\usepackage{amssymb,amsmath}
\usepackage{amsfonts,amsbsy,bm}
\usepackage{latexsym}
\usepackage{exscale}
\usepackage{mathrsfs}

\usepackage{color}
\usepackage[active]{srcltx} 

\newcommand{\N}{\mathbb N}

\newcommand{\ff}{{\mathbf f}}
\newcommand{\be}{{\mathbf e}}

\newcommand{\bk}{{\mathbf k}}

\newcommand{\bc}{{\mathbf c}}
\newcommand{\bz}{{\mathbf 0}}
\newcommand{\tbe}{{\tilde\be}}
\newcommand{\bx}{{\mathbf x}}

\newcommand {\SC} {{\mathbb C}}

\newcommand {\SN} {{\mathbb N}}
\newcommand {\SR} {{\mathbb R}}

\newcommand {\SX} {{\mathbb X}}
\newcommand {\SY} {{\mathbb Y}}
\newcommand {\SZ} {{\mathbb Z}}

\newcommand {\al} {{\alpha}}
\newcommand {\dt} {{\delta}}

\newcommand {\e} {{\varepsilon}}
\newcommand {\ga} {{\gamma}}
\newcommand {\Ga} {{\Gamma}}
\newcommand {\la} {{\lambda}}
\newcommand {\La} {{\Lambda}}

\newcommand {\tK} {{\widetilde K}}

\newcommand {\tsigma} {{\tilde\sigma}}

\def\supp{\mathop{\rm supp}}

\numberwithin{equation}{section}
\newtheorem{theorem}{Theorem}[section]
\newtheorem{lemma}[theorem]{Lemma}

\newtheorem{Remark}[theorem]{Remark}
\newtheorem{proposition}[theorem]{Proposition}

\newtheorem{example}[theorem]{Example}

\newcommand {\Proofof}[1] {\noindent{\bf P{\footnotesize\bf ROOF} of {#1}: } \ }

\newcommand {\ProofEnd} {
             \begin{flushright} \vskip -0.2in $\Box$ \end{flushright}}

\newcommand{\Ba}[1]{\begin{array}{#1}}
\newcommand{\Ea}{\end{array}}
\newcommand{\Be}{\begin{equation}}
\newcommand{\Ee}{\end{equation}}
\newcommand{\Bea}{\begin{eqnarray}}
\newcommand{\Eea}{\end{eqnarray}}
\newcommand{\Beas}{\begin{eqnarray*}}
\newcommand{\Eeas}{\end{eqnarray*}}
\newcommand{\Benu}{\begin{enumerate}}
\newcommand{\Eenu}{\end{enumerate}}
\newcommand{\Bi}{\begin{itemize}}
\newcommand{\Ei}{\end{itemize}}

\newcommand{\BR}{\begin{Remark} \em}
\newcommand{\ER}{\end{Remark}}
\newcommand{\BE}{\begin{example} \em}
\newcommand{\EE}{\end{example}}

\newcommand {\Ts} {\textstyle}

\newcommand {\mand} {{\quad\mbox{and}\quad}}
\renewcommand {\mid} {{\,\,\,\colon\,\,\,}}

\newcommand{\bline}{{\bigskip

\noindent}}

\newcommand{\sline}{{\smallskip

\noindent}}

\newcommand {\bone} {{\bf 1}}

\newcommand {\tA} {{\widetilde A}}

\newcommand {\AG} {{A\setminus\Ga}}
\newcommand {\GA} {{\Ga\setminus A}}

\newcommand {\PG} {{P\setminus\Ga}}
\newcommand {\GP} {{\Ga\setminus P}}
\newcommand {\PB} {{P\setminus B}}

\newcommand {\koj} {{k_1\mbox{\tiny$(j)$}}}

\newcommand {\tG} {{\widetilde \Ga}}

\renewcommand {\span} {\mbox{\rm span}\,}

\newcounter{reg}
\newcounter{regTO}
\begin{document}

\title[Lebesgue-type inequalities for
quasi-greedy bases]{Lebesgue-type inequalities for quasi-greedy
bases}

\author{Gustavo Garrig\'os}

\address{Gustavo Garrig\'os
\\
Departamento de Matem\'aticas
\\
Universidad de Murcia
\\
30100 Murcia, Spain} \email{gustavo.garrigos@um.es}

\author{Eugenio Hern\'andez}

\address{Eugenio Hern\'andez
\\
Departamento de Matem\'aticas
\\
Universidad Aut\'onoma de Madrid
\\
28049, Madrid, Spain} \email{eugenio.hernandez@uam.es}

\author{Timur Oikhberg}

\address{Timur Oikhberg
\\
Department of Mathematics
\\
University of Illinois
\\
Urbana, IL, USA} \email{oikhberg@illinois.edu}

\begin{abstract}
We show that for quasi-greedy bases in real or complex Banach spaces
the error of the thresholding greedy algorithm of order $N$ is
bounded by the best $N$-term error of approximation times a function
of $N$ which depends on the democracy functions and the quasi-greedy
constant of the basis. If the basis is democratic this function is
bounded by $C\log N$. We show with two
examples that this bound is attained for quasi-greedy
democratic bases.
\end{abstract}

\thanks{First and second authors supported by Grant MTM2010-16518 (Spain).
Third author supported by a travel grant from Simons Foundation,
and by a COR grant from University of California System.}

\date{\today}


\maketitle

%
%

\section{Introduction}\label{secIntroduc}

Let $(\SX,\|.\|)$ be a Banach space (real or complex) and $\mathcal{B} = \{\be_j\}_{j=1}^\infty$ a
countable normalized basis\footnote{We assume normalization, $\|\be_j\|=1$,
 for notational convenience; all the results are actually valid for seminormalized bases,
 perhaps after suitable modifications in the constants.}.
Let $\Sigma_N\,, N=1,2,3,\dots$ be the set of all $y\in \SX$ with at
most $N$ non-null coefficients in the unique basis representation.
For $x\in \SX$, the {\bf $N$-term error of approximation} with
respect to $\mathcal B$ is
$$
  \sigma_N(x) = \sigma_N(x;\mathcal B, \SX) :=
  \inf_{y\in\Sigma_N} \|x-y\|_{\SX}\,, \quad N=1,2,3, \dots
$$
Given $x=\sum_{k=1}^\infty a_k(x)\be_k \in \SX\,,$ let $\pi$
denote any bijection of $\mathbb N$ such that
\begin{eqnarray} \label{greedy}
  |a_{\pi(k)}(x)| \geq |a_{\pi(k+1)}(x)|\, \quad \mbox{for all} \
  k\in \mathbb N\,.
\end{eqnarray}
The {\bf thresholding greedy algorithm of order $N$} (TGA) is
defined by
$$
  G_N(x) = G_N^\pi (x;\mathcal B, \SX) := \sum_{k=1}^N a_{\pi(k)}(x)
  \be_{\pi(k)}\,.
$$
It is not always true that $G_N(x) \to x$ (in $\SX$) as $N \to
\infty\,.$ A basis $\mathcal B$ is called {\bf quasi-greedy} if
$G_N(x) \to x$ as $N \to \infty$ for all $x\in \mathbb
X\,.$ This turns out to be equivalent (see \cite[Theorem 1]
{Wo}) to the existence of some constant $\tK$ such that
\begin{eqnarray}\label{1.4}
  \sup_N \|G_N(x)\| \leq \tK \| x \|\quad
  \mbox{for all} \ x\in \SX\,.
\end{eqnarray}
We define the {\bf quasi-greedy constant} $K$  of the basis
$\mathcal B$ to be the least $\tK$ such that \eqref{1.4} holds for
all permutations $\pi$ satisfying \eqref{greedy}.

Given a basis $\mathcal B$ in a Banach space $\SX$, a {\bf
Lebesgue-type inequality} for the TGA is an inequality of the form
$$
  \|x- G_N(x)\| \leq
  C(N)\,\sigma_N(x)\,,\quad   x\in \SX\,,
$$
where $C(N)$ is a nondecreasing function of $N$. For a survey on
Lebesgue-type inequalities for the greedy algorithm see \cite{T1,
T2} and the references therein. We specially mention the recent
papers \cite{TYY2011a,TYY2011b}, which deal with Lebesgue-type
inequalities for quasi-greedy bases in $L^p$ spaces (see also
\cite{DST}).

The purpose of this paper is to study such inequalities for
quasi-greedy bases in general Banach spaces, thus complementing
and in some cases improving the results in
\cite{TYY2011a,TYY2011b,DST}. Towards this end we define the
sequence
\[
C_N:=\sup_{x\in\SX}\frac{\|x-G_N(x)\|}{\sigma_N(x)}\,.
\]
Following the notation in previous papers, we write \[
h_l(n)=\inf_{|A|=n}\big\|\sum_{j\in A}\be_j\big\|,\quad
h_r(n)=\sup_{|A|=n}\big\|\sum_{j\in
A}\be_j\big\|\mand\mu(N)=\sup_{n\leq N}\frac{h_r(n)}{h_l(n)}.\]
These functions are implicit in the first works on $N$-term
approximation. For instance, $\mu(N)$ is defined in \cite{Wo},
and $h_l$, $h_r$ appear explicitly in \cite{KaT}. In \cite{GHM2008,GHN} the latter are called
left and right democracy functions of the basis $\mathcal B$.

For $A\subset\SN$, we
 denote by $S_A$
the projection operator
\[ \Ts x=\sum_{j=1}^\infty a_j\be_j\longmapsto S_A(x)=\sum_{j\in A} a_j\be_j\;,\]
and consider also the sequence
\[
k_N:=\sup_{ |A|\leq N}\|S_A\|.\]

Our main result is the following:

\begin{theorem}\label{Th1}
If $\{\be_j\}_{j=1}^\infty$ is a (normalized) quasi-greedy basis in a
Banach space $\SX$ (real or complex), then \Be
C_N\,\approx\,\max\{\mu(N), k_N\},\quad \forall\;N=1,2,...
\label{1}\Ee
\end{theorem}

\sline\emph{Remarks:}

 \Benu
\item  When $\{\be_j\}$ is unconditional, then
$k_N=O(1)$, so we obtain as a special case Theorem 4 in
\cite{Wo}.

\item For quasi-greedy bases it can be shown that \Be k_N \leq
c\log N.\label{kNlogN}\Ee This is essentially contained in
\cite[Lemma 8.2]{DKK} (see also \cite[Lemma 2.3]{DST}). Since this result
is often used in the paper, we outline a proof in $\S5$ below.

\item When $\{\be_j\}$ is quasi-greedy, some upper bounds for $C_N$
have recently appeared in the literature: in \cite[Theorem 2.1]{TYY2011b} it was shown that\Be
 C_N\,\lesssim\,\mu(N)\,k_N\,,\label{kNTYY}
 \Ee while in  \cite[Thm 1.1]{H} it is proved that
\Be C_N\,\lesssim\,\sum_{1\leq k\leq N}\frac{\mu(k)}{k}\quad
\Big(\mbox{\small$\lesssim\,\mu(N)\,\log N$}\;\Big).\label{kN} \Ee
Notice that \eqref{1}, being an equivalence, improves strictly over these in some cases.
 For instance, if $\SX$ is such that $\mu(N)\approx (\log N)^\al$ and say $k_N\approx\log N$,
 then  \eqref{kNTYY} and \eqref{kN} would only give
$C_N\lesssim(\log N)^{\al+1}$, while Theorem \ref{Th1} implies
$C_N\approx(\log N)^{\min\{\al, 1\}}$. For constructions of such
examples, see \eqref{6kN} below.

\item When $\{\be_j\}$ is quasi-greedy and democratic (i.e. $\mu(N)=O(1)$), then \eqref{1} and \eqref{kNlogN} give \Be
C_N\,\approx\, k_N\,\lesssim \,\log N. \label{aux2}\Ee We show in
section 6 that this logarithmic bound can actually be attained, answering a question posed in \cite{H}.
 One such
example is given by the Haar basis in $BV(\SR^d)$, $d>1$ (see
$\S$\ref{examples}).
This is in contrast with the Hilbert space case, where it was
recently noticed by Wojtaszczyk that $k_N$ cannot attain $\log N$
(\cite{Wo4}; see also $\S9$ below).

 \Eenu

 \medskip

Let also denote by $\tilde\sigma_N(x)$ the \emph{expansional best
approximation} to $x$, that is if $x=\sum_{k=1}^{\infty} a_k
\be_k\,,$ then
$$ \tilde \sigma_N (x) = \tilde \sigma_N (x; \mathcal B, \SX) :=
\inf_{A,\, |A|=N} \{\|x - \sum_{k\in A} a_k \be_k\|\}.$$ In this
case it is known that, for quasi-greedy bases,\Be
\widetilde{C}_N:=\sup_{x\in\SX}\frac{\|x-G_N(x)\|}{\tilde\sigma_N(x)}\,\approx\,\mu(N);
\label{tCN}\Ee see \cite[Theorem 2.2]{TYY2011b} for the upper
bound (the lower bound was essentially in \cite{Wo}; see also
Proposition \ref{proCmu} below). In \cite{H} it was asked whether
one could prove bounds for $C_N$ using \eqref{tCN} and suitable
bounds on the sequence
\[D_N := \sup_{x\in\SX}\frac{\tilde\sigma_N(x)}{\sigma_N(x)}\geq 1.\]
Here we prove the following

\begin{theorem}\label{Pro2}
For any (normalized, not necessarily quasi-greedy) basis $\{\be_j\}$
we have \Be \frac{k_N}4\,\leq\,D_N\,\leq\, 2 k_N,\quad
\forall\;N=1,2,... \label{5}\Ee
\end{theorem}

\sline\emph{Remark:}
The right hand side of \eqref{5} together with \eqref{kNlogN}
gives $\tsigma_N(x) \lesssim(\log N)\sigma_N(x)$ for quasi-greedy bases. This was noticed in
\cite[Lemma 2.4]{DST}, answering a
question from \cite{H}. The left hand side of \eqref{5} seems to be new.

\

Our last result is the following theorem, which
answers a question of
Wojtaszczyk (personal communication to the second author on
November, 2011).

\begin{theorem}
\label{th2new} If $\{\be_j\}_{j=1}^\infty$ is a quasi-greedy basis in $\SX$, then there exists $c>0$ such
that for all $N,k=1,2,\ldots$
\[ 
\big\|x-G_{N+k}(x)\big\|\,\leq\,c\,\Big(1+
\frac{h_r(N)}{h_l(k)}\Big)\,\sigma_N(x),\quad \forall\;x\in\SX.
\] 
\end{theorem}

\sline Results of this type have appeared before in
the literature. For unconditional bases, this theorem was proved in
\cite[Thm 5]{KaT}; see also \cite[Thm 4]{Wo3}. For quasi-greedy
democratic bases it is essentially contained in \cite{DKKT}. Here
we extend its validity to general quasi-greedy bases.


A slightly weaker version of Theorem \ref{th2new} can also be found in \cite{KT2}; namely, given $N$ and $k$ there
exists a set $A$ of cardinality not exceeding $N+k$ such that
$$
\big\|x-S_{A}(x)\big\|\,\leq\,c\,\Big(1+
\frac{h_r(N)}{h_l(k)}\Big)\,\sigma_N(x),\quad \forall\;x\in\SX. $$
The improvement in Theorem \ref{th2new} consists in showing that the set $A$ can be obtained by
running the greedy algorithm.

 We finally remark that the proofs of Theorems \ref{Th1}, \ref{Pro2} and \ref{th2new} combine ideas
 present in various of the above quoted references, but whose main lines essentially stem from the original work
 of Konyagin and Temlyakov \cite{KT}.

 \

{\bf Acknowledgements}. This work started when the second and third
authors participated in the \textit{Concentration week on greedy
algorithms in Banach spaces and compressed sensing} held on July
18-22, {{2011,}} at Texas A\&M University. We
 express our gratitude to the Organizing Committee for
the invitation to participate in this meeting. In addition, the
third author thanks the second author for arranging his visit to
UAM, where this work continued. A preliminary version of this paper
(\cite{H}) was written by the second author and posted in ArXiv in
November 2011.
\medskip

\section{Proof of Theorem \ref{Th1}: upper bounds}

The proof follows the strategy developed in \cite{KT}, together
with two known estimates for quasi-greedy bases. First, as
mentioned in $\S1$, there exists a (smallest) constant $K$ such
that \Be
\|G_N(x)\|\leq\,K\,\|x\|,\quad\forall\;x\in\SX,\;\;\forall\;N=1,2,\ldots,\label{qg-constant}
\Ee see \cite[Th 1]{Wo}. Also, there exist $c_1,c_2>0$ such that
\Be c_1 (\min_{k\in A}|a_k|)\,\|\sum_{k\in A}\be_k
\|\leq\|\sum_{k\in A}a_k\be_k\|\leq c_2 (\max_{k\in
A}|a_k|)\|\sum_{k\in A}\be_k\|.\label{dil}\Ee These inequalities
are proved in \cite[Lemmas 2.1 and 2.2]{DKKT} for real scalars
$a_k$, setting $c_1= 1/(4K^2)$ and $c_2=2K$. For completeness, in the
appendix ($\S10$) we outline the proof also for complex scalars $a_k$, in
which case one can let $c_1= 1/(8\sqrt{2}K^2)$ and
$c_2=4\sqrt{2}K$.

\

We shall write $a^*_k(x)$ for the decreasing rearrangement of the
basis coefficients of $x$; that is, if $x=\sum_{j=1}^\infty
a_j\be_j$, we set $a^*_k(x)=|a_{\pi(k)}|$ when $\pi$ is any
permutation of $\SN$ such that $|a_{\pi(1)}|\geq|a_{\pi(2)}|\geq
|a_{\pi(3)}|\geq ... $ As in \cite{H}, we shall use the following
simple (but crucial) observation.

\begin{lemma}
\label{L1} For all $A\subset\SN$ and $x\in\SX$ we have
$a^*_k\big(S_A(x)\big)\leq a^*_k(x)$.
\end{lemma}

We now prove the theorem. Fix $N\geq 1$ and $x\in\SX$. Take any
$y=\sum_{j\in A}y_j\be_j$ with $|A|=N$. We shall show that \Be
\|x-G_N(x)\|\leq\,c\,\max\{\mu(N), k_N\}\,\|x-y\| . \label{goal}\Ee
Then, taking the infimum of the right hand side over all $y\in\Sigma_N$,
we obtain the upper estimate for $C_N$ in \eqref{1}.

Write $G_N(x)=S_\Ga(x)$ with $|\Ga|=N$. Then \Beas \|x-G_N(x)\| &
= &  \|x-S_A(x)+S_A(x)-S_\Ga(x)\|\\
& \leq & \|x-S_A(x)\|+ \|S_{\AG}(x)\|+\|S_{\GA}(x)\|.\Eeas The
first and third terms are easily bound by $c\,k_N\|x-y\|$; namely,
\[ \|x-S_A(x)\| \leq  \|x-y\|+\|S_A(y)-S_A(x)\|\leq
(1+k_N)\|x-y\|,\] while \[\|S_\GA(x)\|  =
\big\|S_{\GA}(x-y)\big\|\,\leq \,k_N\,\|x-y\|.
\]
We next show that $\|S_\AG(x)\|$ can be controlled by
$c\,\mu(N)\|x-y\|$. First notice that \Be \big\|S_\AG(x)\big\|\,\leq
\, c_2(\max_{k\in\AG}|a_k|)\,{\big\|\sum_{k\in\AG} \be_k
\big\|}\,\leq\, c_2\,\mu(N)\,(\min_{k\in\GA}|a_k|)\,{
\big\|\sum_{k\in\tG} \be_k\big\|}\, \label{aux1}\Ee
where we choose as $\tG$ any set of cardinality $|\tG|=|\AG|=|\GA|$
in which $x-y$ attains the largest coefficients, i.e.
$G_{|\GA|}\big(x-y\big)=S_\tG(x-y)$. From Lemma \ref{L1} one easily
sees that \Beas \min_{k\in\GA}\big|a_k(x)\big| & =&
\min_{k\in\GA}\big|a_k\big(S_\GA(x-y)\big)\big|\,=\,a^*_{|\GA|}\big(S_{\GA}(x-y)\big)\\&
\leq& a^*_{|\GA|}(x-y) \,=\,\min_{k\in \tG} \big|a_k(x-y)\big|. 
\Eeas Thus, using again \eqref{dil}, one can bound the right side
of \eqref{aux1} by a constant times
\[ \mu(N)\,\big\|S_\tG(x-y)\big\|\,=\,\mu(N)\,\big\|G_{|\GA|}(x-y)\big\|\,\leq\, K\,\mu(N)\,\|x-y\|,
\]
as we wished to prove. Notice that the final multiplicative constant
involved in this process is of order $Kc_2/c_1=O(K^4)$. \ProofEnd

\section{Proof of Theorem \ref{Th1}: lower bounds}

The bound $C_N\gtrsim \mu(N)$ was proved by Wojtaszczyk when
$\{\be_j\}$ is an unconditional basis; \cite[Thm 4]{Wo}. As pointed
out in \cite{H}, these arguments can easily be adapted to the more
general setting of quasi-greedy bases; we include the proof for
completeness. Recall that $K$ is our notation for the quasi-greedy
constant defined in \eqref{qg-constant}.

\begin{proposition}\label{proCmu}
If $\{\be_j\}$ is quasi greedy then $C_N\geq \,\widetilde{C}_N\geq\tfrac1{3K}\, \mu(N)$.
\end{proposition}

We shall use the following lemma. Here we denote $\bone_A=\sum_{j\in A}\be_j$.
\begin{lemma}\label{3.2}
If $\{\be_j\}$ is quasi greedy with constant $K$, then for every $N$
there exist disjoint sets $A,B$ such that
\[
|A|=|B|\leq N\mand \frac{\|\bone_A\|}{\|\bone_{B}\|}\,\geq
\tfrac1{3K}\,\mu(N).
\]
\end{lemma}

 \begin{proof}
We may assume that $\mu(N)>3K$ (otherwise choose $|A|=|B|=1$). Then
there exist $A,B$ (not necessarily disjoint) with $|A|=|B|\leq N$
and
\[
\max\{\tfrac12\mu(N),\,3K\}\,<\, \frac{\|\bone_A\|}{\|\bone_{B}\|}.
\]
The quasi-greedy condition implies that $\|\bone_{A\cap B}\|\leq K
\|\bone_{B}\|$, which inserted above gives
\[
3K< \frac{\|\bone_A\|}{\|\bone_{B}\|}\leq
K\,\frac{\|\bone_A\|}{\|\bone_{A\cap B}\|}
\]
and therefore $\|\bone_{A\cap B}\|\leq \frac13 \|\bone_{A}\|$.
Thus
\[\frac{\|\bone_A\|}{\|\bone_{B}\|}\,\leq \frac{\|\bone_{A\cap B}\|+\|\bone_{A\setminus B}\|}{\|\bone_{B}\|}\,\leq\,
\tfrac 13\,\frac{\|\bone_A\|}{\|\bone_{B}\|} \,+\,\frac{\|\bone_{A\setminus B}\|}{\|\bone_{B}\|}\,,\]
which can be rewritten as
\[
\frac{\|\bone_A\|}{\|\bone_{B}\|}\,\leq \,\tfrac32\,\frac{\|\bone_{A\setminus B}\|}{\|\bone_{B}\|}.
\]
Now set $\tA:=(A\setminus  B)\cup C$, for any  $C$, disjoint with
$B$, such that $|\tA|=|B|$. Then \[ \tfrac12\,{\mu(N)}\,\leq
\frac{\|\bone_A\|}{\|\bone_{B}\|}\,\leq
\,\tfrac32\,\frac{\|\bone_{A\setminus B}\|}{\|\bone_{B}\|}\,\leq
\,\tfrac32K\,\frac{\|\bone_{\tA}\|}{\|\bone_{B}\|},\] which gives
the desired result since  $\tA\cap B=\emptyset$.
\end{proof}

\Proofof{Proposition \ref{proCmu}}
Consider sets $A$ and $B$ as in the lemma,
and take any set $C$, disjoint with $A\cup B$, such that $|C|=N-|A|=N-|B|$.
Choosing $x=(1+2\e)\bone_B+ (1+\e)\bone_C+\bone_A$ we have
\[
\|x-G_N(x)\| \,=\,\|\bone_A\|\geq
\tfrac1{3K}\,\mu(N)\,\|\bone_B\|\geq
\tfrac1{(1+2\e)3K}\,\mu(N)\tilde\sigma_N(x),\] which proves the result
when $\e\to 0$.\ProofEnd

To establish the lower bound in Theorem \ref{Th1} it remains to show the following.

\begin{proposition}\label{Pro4}
For any basis (not necessarily quasi-greedy) we have \Be C_N\geq k_N/4.\label{4}\Ee
\end{proposition}

 \begin{proof}
Assume that $k_N\geq4$ (otherwise \eqref{4} is trivial). For fixed
$N$ find $A\subset\SN$ with $|A|\leq N$ and $x\in\SX$ such that
$\|S_A(x)\|>(k_N/2)\|x\|$. We may assume that $x=\sum_{j\in
B}x_j\be_j$ with $B$ finite and $\|x\|=1$. Note that
$\|x-S_A(x)\|\geq \|S_A(x)\|-\|x\|\geq k_N/4$.

Take any number $r>\max|x_j|$, and set $y=x-S_A(x)+r\bone_\tA$.
Here $\tA$ is any set of cardinality $N$ containing $A$ and (if
necessary) some indices in $B^c$. Then
\[
\|y-G_N(y)\|=\|x-S_A(x)\|\geq k_N/4.\] On the other hand, since
$r\bone_\tA-S_A(x)\in\Sigma_N$ we have
\[
\sigma_N(y)\leq \|y-(r\bone_\tA-S_A(x))\|=\|x\|=1,\] which gives
\eqref{4}.
\end{proof}

\section{Proof of Theorem \ref{Pro2}}

The upper bound  $D_N\leq 2k_N$ is elementary. Indeed, let $x\in\SX$
and $p = \sum_{k\in P} b_k \be_k\in\Sigma_N$ with $|P|=N$. Then,
\[ \tilde \sigma_N (x) \leq \|x- S_P(x)\| \leq \|x-p\| +
\|S_P(x-p)\|\,\leq \,(1+k_N)\|x-p\|.\] Taking the infimum over all $p\in\Sigma_N$
we obtain\[ \tilde \sigma_N (x) \leq (1+k_N)\sigma_N(x)\,,\]
which proves $D_N\leq 2k_N$.

For the converse we argue as in the proof of Proposition \ref{Pro4}.
That is, we choose an  $x=\sum_{j\in B}x_j\be_j$ with $\|x\|=1$, and
a set $A$ so that $\|x-S_A(x)\|\geq k_N/4$, and we let
$y=x-S_A(x)+r\bone_\tA$ as before. This time we shall choose
$r>(2+k_N+k_{2N})2\bc$, where $\bc$ is the basis constant, and we
shall prove that, with this choice
\[
\tsigma_N(y)\,\geq\, \tfrac{k_N}4\,\sigma_N(y),
\]
which clearly implies $D_N\geq k_N/4$.

As shown before, $\sigma_N(y)\leq1$, so we need to prove that
\Be
\tsigma_N(y)=\inf_{|C|\leq N}\|y-\sum_{(B\setminus A)\cap C}x_j\be_j-r\bone_{\tA\cap C}\|\geq k_N/4.
\label{aux5}\Ee
Suppose we are given one such set $C$ which is not equal to $\tA$. Then there must be some
$j_0\in\tA\setminus C$, and we would have
\Beas
\|y-\sum_{(B\setminus A)\cap C}x_j\be_j-r\bone_{\tA\cap C}\| &  = &
\|\sum_{B\setminus (A\cup C)}x_j\be_j+r\bone_{\tA\setminus C}\|\\
& \geq & r\,\big\|\bone_{\tA\setminus C}\big\|\,-\, \big\|\sum_{B\setminus (A\cup C)}x_j\be_j\big\|\\
& \geq & \tfrac r{2\bc} \|\be_{j_0}\| \,-\, \big\|x-S_{A\cup
C}(x)\big\|\,\geq \, \tfrac r{2\bc}\,-\, (1+k_{2N})>1+k_N,\Eeas
where in the third line we have used that $e_{j_0} =
P_{j_0}(\bone_{\tA\setminus C}) - P_{j_0 -1}(\bone_{\tA\setminus
C})$, and the partial sums operators $P_j$ have norm bounded by
$\bc$ . On the other hand, if we use $C=\tA$ we obtain a better
estimate
\[
\big\|y-r\bone_{\tA}\big\|=\big\|x-S_A(x)\big\|\,\leq\,1+k_N.\]
Therefore,
\[
\tsigma_N(y)=\big\|y-r\bone_{\tA}\big\|=\big\|x-S_A(x)\big\|\,\geq
\,k_N/4\,,\]
proving \eqref{aux5}.

\section{An upper bound for $k_N$}

We prove a bound for the constants
$k_N$ when $\{\be_j\}$ is a quasi-greedy basis.

\begin{theorem}\label{Th3}
If the basis is quasi-greedy, there exists $c>0$ such that \Be k_N
\leq c \log N,\quad \forall\;N=2,3,... \label{3}\Ee
\end{theorem}

This was essentially shown in \cite[Lemma 8.2]{DKK} (see also \cite[Lemma 2.3]{DST}),
but we include a self-contained proof for completeness.
We need two easy lemmas.

\begin{lemma}\label{L4.1}
 Let $(\SX, \{\be_j\})$ be quasi-greedy. Consider $x=\sum_{i} a_i
 \be_i \in \SX$ and $0 \leq \alpha < \beta < \infty\,.$ Let $F=\{i:
 |a_i| \in (\alpha, \beta]\}.$ Then $\|S_F(x)\| \leq 2K \|x\|\,.$
\end{lemma}

\begin{proof}
Let $G=\{i:  |a_i|> \alpha \}$ and  $H=\{i:  |a_i| > \beta\}.$ By
the definition of quasi-greediness, $\max\{\|S_G(x)\|, \|S_H(x)\|\}
\leq K \|x\|\,.$ However, $S_F(x)= S_G(x) - S_H(x)\,.$ Apply the
triangle inequality to finish the proof.
\end{proof}

\begin{lemma}\label{L4.2}
 Let $(\SX, \{\be_j\})$ be quasi-greedy. Consider $x=\sum_{i} a_i
 \be_i \in \SX$ and $0 < \alpha < \beta < \infty\,.$ Then, for any $P \subset F=\{i:
 |a_i| \in (\alpha, \beta]\},$ we have
 \[\|S_P(x)\| \leq \,K\,\frac{c_2}{c_1}\,\frac{\beta}{\alpha}\, \|S_F(x)\|\,,\]
where $c_1, c_2$ are as in \eqref{dil}.
\end{lemma}

\begin{proof}
We use (\ref{dil}) (see \cite{DKKT} or Proposition \ref{p1.ap}
below). We have
\[\|S_P(x)\| = \|\sum_{i\in P} a_i \be_i \| \leq c_2\beta
\|\sum_{i\in P} \be_i \|\,.\] By quasi-greediness \[ \|\sum_{i\in
P} \be_i \| \leq K \|\sum_{i\in F} \be_i \|\,.\] Finally, \[
\|\sum_{i\in F} \be_i \| \leq \frac{1}{c_1\alpha} \|\sum_{i\in F}
a_i \be_i \| = {\frac{1}{c_1\alpha}} \| S_F(x)\|\,.\]
\end{proof}

We now prove Theorem \ref{Th3}. Take $|A|=N\geq 2\,.$ Let
$x=\sum_i a_i \be_i\,.$ By scaling we may assume $\max_i |a_i| =
1$. Under this assumption
 \Be \|x\| \geq (1/K). \label{4.1}\Ee In fact, if $\max_i |a_i| = |a_{i_0}| =
1$ for some $i_0$, then $ 1=|a_{i_0}| \|\be_{i_0}\| =\|G_1(x)\|
\leq K\|x\|$, proving (\ref{4.1}).

Let $\ell\in \SN$ so that $2^{-\ell} \leq \frac{1}{N} <
2^{1-\ell}.$ Represent $A$ as a disjoint union of sets
$\cup_{k=1}^\ell A_k$, where $A_k = \{i\in A: 2^{-k} <|a_i| \leq
2^{1-k}\}$, {\small $1\leq k \leq \ell -1$}, and $A_\ell = \{i\in
A: |a_i| \leq 2^{1-\ell}\}$. Then, by (\ref{4.1}),
 \[\|S_{A_\ell}x\| \leq \sum_{i\in A_\ell} |a_i| \|\be_i\| \leq
 2^{1-\ell} |A_\ell| \leq \frac{2|A_\ell| }{N} \leq 2 \leq 2K
 \|x\|.\]
For $1\leq k \leq \ell -1$, let $F_k = \{i\in \SN: 2^{-k} <|a_i|
\leq 2^{1-k}\}.$ By Lemmas \ref{L4.2} and \ref{L4.1}
 \[ \|S_{A_k} (x)\| \leq \,c'\, \|S_{F_k}(x)\| \leq \,c''\, \|x\|,\]
with $c''=4K^2c_2/c_1$.  Therefore,
 \[ \|S_{A} (x)\| \leq  \sum_{k=1}^\ell \|S_{A_k} (x)\| \leq
 (2K + c'' (\ell-1))  \|x\|.\]
As $\ell - 1 \leq \log_2 N$, we have shown \eqref{3} with $c$ of the
order $K^2c_2/c_1=O(K^5)$.

\section{Examples}
\label{examples}

We compute (asymptotically) the Lebesgue-type constants $C_N$ for some
explicit examples of quasi-greedy democratic bases. Notice that,
in view of Theorem \ref{Th1}, for such bases we have
\[
C_N\,\approx\,k_N =\sup_{|A|\leq N}\|S_A\|.
\]

\bline{\bf Example 1:} \emph{the Lindenstrauss basis}.
 Consider the system of vectors in $\ell^1$ defined by\[
\bx_n=\be_n\,-\tfrac12\big(\be_{2n}+\be_{2n+1}\big),\quad
n=1,2,\ldots
\]
where $\be_n$ denotes the canonical basis. It is known that
$\{\bx_n\}_{n=1}^\infty$ is a monotone basic sequence in $\ell^1$,
and a conditional basis in its closed linear span $D$; see e.g.
\cite[p. 27]{LZ} or \cite[p. 455]{Sin}. The space $D$ was introduced
by J. Lindenstrauss in \cite{L} and has other interesting properties
in functional analysis. In particular, it was shown by Dilworth and
Mitra \cite{DM} that $\{\bx_n\}_{n=1}^\infty$ is a quasi-greedy
basis in $D$.

Here we show that
\[
k_N\,\approx \, \log N,
\]
which in particular gives a direct proof that the Lindenstrauss
basis is not unconditional. By Theorem \ref{Th3}, it suffices
to show the lower bound. We first notice that
\[
\big\|\sum_{k=1}^n
b_k\bx_k\big\|_{\ell^1}\,=\,|b_1|\,+\,\sum_{k=2}^n\big|b_k-\tfrac12\,b_{\lfloor
\frac
k2\rfloor}\big|\,+\,\frac12\,\sum_{k=n+1}^{2n+1}\big|b_{\lfloor
\frac k2\rfloor}\big|.\] Now consider
\[
\Ts \bx= \sum_{j=0}^{n-1}\sum_{2^j\leq k<2^{j+1}}2^{-j}\bx_k\,.
\]
Clearly,
\[
\|\bx\|_{\ell^1}\,=\,1\,+\,\sum_{j=1}^{n-1}\sum_{2^j\leq
k<2^{j+1}}\big|2^{-j}-\tfrac12\,2^{-(j-1)}\big|\,+\,\sum_{2^n\leq
k<2^{n+1}}2^{-n}\,=\,2.\] Now choose $A=\cup_{{0\leq j<n}\atop{j\;
{\rm even}}}[2^j,2^{j+1})\cap\mathbb{N}$, so that $N:=|A|\approx
2^n$. Then, if say $n$ is odd,
\[
\big\|S_A\bx\big\|_{\ell^1}\,=\,1\,+\,\sum_{j=1}^{n}\sum_{2^j\leq
k<2^{j+1}}2^{-j}\,=\,n+1\,\approx\,\log N.\] Thus
$k_N\geq\|S_A\|\gtrsim\log N$, proving our claim.

 \

{\bf Example 2.} An important example of quasi-greedy basis arises
in the context of $BV(\SR^d)$, $d>1$. This space is not separable,
so we consider the closed linear span $\SX$ of the $d$-dimensional
(non-homogeneous) Haar system in the $BV$-norm
\[
\|f\|_{BV(\SR^d)}=\|f\|_{L^1(\SR^d)}+|f|_{BV(\SR^d)},
\]
where $|f|_{BV}$ is the total variation of the distributional gradient $\nabla f$ (as defined e.g. in \cite[(1.1)]{CDDD}).
It follows from the
results in \cite{CDPX,Wo2} that the Haar system is a quasi-greedy
democratic basis in $\SX$ (see e.g. \cite[Thm
10]{Wo2})\footnote{Democracy is not explicitly stated, but follows
easily from the inclusions $\ell^1\hookrightarrow BV
\hookrightarrow\ell^{1,\infty}$ as in \cite[p. 239]{CDDD}. The
fact that the Haar system is a basic sequence in $BV$ (hence a
basis in its closed linear span $\SX$), is a consequence of the
uniform boundedness of the projections, see \cite[Corollary
12]{Wo2}. Finally, it is a seminormalized system with the
normalization in \eqref{hjk}; see \cite[(1.6)]{CDDD}.}. We claim
that in this case
\[ k_N \approx \log N.
\]
It suffices to show the lower bound. For this we will argue as in
\cite{bechler}, to find functions $f_N\in\Sigma_{2N}$ with
$\|f_N\|_{BV(\SR^d)}=O(1)$, and sets $A_N$ with $|A_N|=N$ such that
$|S_{A_N}(f_N)|_{BV} \geq c\, \log N$.

To do this carefully we first set some notation. The Haar
functions are defined by \Be
h^\be_{j,\bk}(x)=2^{j(d-1)}\prod_{\ell=1}^d
h^{e_\ell}(2^jx_\ell-k_\ell), \quad j\geq 0,
\;\bk\in\SZ^d,\;\be\in\{0,1\}^d,\;x\in\SR^d,\label{hjk}\Ee where
$h^0=\chi_{[0,1)}$ and
$h^1=\chi_{[0,\frac12)}-\chi_{[\frac12,1)}$. With this definition
the Haar system  is semi-normalized, i.e.
$c_1<\|h^\be_{j,\bk}\|_{BV(\SR^d)}<c_2$. The (non-homogeneous)
Haar system is obtained restricting to indices $\la=(j,\bk,\be)$
with $\be\not=\bz$ when $j>0$. We sometimes write it
$\mathcal{H}=\{h_\la\}_{\la\in\La}$. As explained above, it is a
quasi-greedy democratic basis in $\SX$, the $\|\cdot\|_{BV}$-closure
of its linear span.

Following \cite{bechler} we consider the function
$f=\chi_{[0,\frac13]\times[0,1]^{d-1}}$ and $f_n=P_{2n}f$, where
$P_J$ denotes the projection onto $V_{J}=\span\{h^\be_{j,\bk}\,|\,
j\leq J\}$. The Haar coefficients of $f$ are easily computed,
leading to the expression\Be f_n=\tfrac13 \chi_{[0,1)^d}
\,+\,\tfrac13\sum_{j=0}^{2n}2^{-j(d-1)}\sum_{{k_2,\ldots,k_d}\atop{0\leq
k_\ell<2^j}} h^{(1,0,\ldots,0)}_{j,\,(k_1(j),k_2,\ldots,k_d)}
\label{fn}\Ee where $\koj$ denotes the only integer such that
$\frac13\in (\frac k{2^{j}},\frac{k+1}{2^{j}})$, explicitly given by
\Be
\koj=\left\{\Ba{ll}\frac{2^j-1}3 & \mbox{ if $j=$ even }\\
\frac{2^j-2}3 & \mbox{ if $j=$ odd.}\Ea\right.\label{k1j}\Ee Using
for instance \cite[Corollary 12]{Wo2} one justifies that
$\|f_n\|_{BV}=\|P_{2n}f\|_{BV}=O(1)$. Note also that
$f_n\in\Sigma_{2N}$ with $N=
O(2^{2n(d-1)})$.

Consider now the set $A_n$ consisting only of the indices in
\eqref{fn} with $j$ even, so that $|A_n|=N$ and\[
S_{A_n}(f_n)=\tfrac13\sum_{{j=0}\atop{j\; {\rm
even}}}^{2n}2^{-j(d-1)}\sum_{{k_2,\ldots,k_d}\atop{0\leq
k_\ell<2^j}} h^{(1,0,\ldots,0)}_{j,\,(k_1(j),k_2,\ldots,k_d)}\,.\]
 To estimate $|S_{A_n}(f_n)|_{BV}$ from below we shall use the
following linear functional
\[
u\in
BV\longmapsto\,\Phi(u)=\int_{[\frac13,\infty)\times\SR^{d-1}}\partial_{x_1}u.
\]
 This is bounded in $BV$ since $\partial_{x_1}u$ defines a
 finite measure.
Thus, \Be |S_{A_n}(f_n)|_{BV} \geq
\big|\Phi\big(S_{A_n}(f_n)\big)\big|. \label{aux4}\Ee On the other
hand, when $(j,\bk,\be)\in A_n$ we can compute explicitly \Beas
\Phi(h^\be_{j,\bk}) &  = &
\int_{[\frac13,\infty)}2^j(h^1)'(2^jx_1-\koj)\,dx_1\\
& = & (h_1)'\,\big[\tfrac{2^j}3-\koj,\infty\big)\\
& = &
\big(\dt_0-2\dt_{1/2}+\dt_{1}\big)\,\big[\tfrac{2^j}3-\koj,\infty\big)\,=\,-1,\Eeas
where in the last step we have used \eqref{k1j} for  $j=$ even.
Thus\[ \big|\Phi\big(S_{A_n}(f_n)\big)\big|=n/3\,\geq\,c\,\log N
\]
which together with \eqref{aux4} proves our assertion.

\

{\bf Example 3.} We now show that $C_N\approx k_N$ may be strictly
smaller than $\log N$. Modifying an example in \cite{KT}, for
$1<p<\infty$ we let $\SX_p $ be the closure of $\span\{\be_j\}$
with the norm \Be |\!|\!|(x_j)|\!|\!|:=\,
\max\Big\{\,\|(x_j)\|_{\ell^{p}}\,,\,
\sup_{m\geq1}\big|\Ts\sum_{n=1}^m
\frac{x_n}{n^{1/p'}}\big|\,\Big\}. \label{bp}\Ee A simple
generalization of the arguments in \cite{KT} shows that the
canonical basis is quasi-greedy and democratic in $\SX_p$. We
claim that, in this example, \Be C_N \approx k_N \approx (\log
N)^{1/p'}. \label{CNbp}\Ee Clearly, for
$x=\sum_{n=1}^\infty\be_n\in\SX_p$
\[\big\|S_Ax\big\|_{\ell^p}\leq \|x\|_{\ell^p}\,\leq\,|\!|\!|x|\!|\!|.\]
Also, if for simplicity we write
$\|x\|_{b^p}:=\sup_{m\geq1}\big|\Ts\sum_{n=1}^m
\frac{x_n}{n^{1/p'}}\big|$, using H\"older's inequality we have
\Beas \big\|S_Ax\big\|_{b^p} & = &
\sup_{m\geq1}\Big|\sum_{{n=1}\atop{n\in A}}^m
\frac{x_n}{n^{1/p'}}\Big|\\ & \leq &
\|x\|_{\ell^p}\,\sup_{m\geq1}\Big(\sum_{{n=1}\atop{n\in
A}}^m\frac1n\Big)^{1/p'} \lesssim
\,|\!|\!|x|\!|\!|\,\big(\log|A|\big)^{1/p'}\,.\Eeas These two
inequalities give the upper bound in \eqref{CNbp}.

On the other hand, testing with
$x=\sum_{n=1}^{2N}\frac{(-1)^n\be_n}{n^{1/p}}$ and
$A=\{1,\ldots,2N\}\cap2\SZ$ one easily sees that
\[
|\!|\!|x|\!|\!| \approx (\log N)^{1/p}\mand
|\!|\!|S_A(x)|\!|\!|\approx \log N.\] This gives $k_N\gtrsim (\log
N)^{1/p'}$, establishing \eqref{CNbp}.

\ \

{\bf Example 4.} Above, we considered examples of quasi-greedy
bases. We provide an example of a non quasi-greedy basis where
\[k_n, C_n, D_n \gtrsim n, \qquad n=1, 2, \dots .\]
 Consider the sequence space $\ell^1$ with the difference
basis
\[
\bx_1 =\be_1,\quad \bx_n=\be_n\,-\be_{n-1},\quad n=2,3,\ldots
\]
Clearly, for finitely supported scalars $(b_n)$, one has
\[
\|\sum_n b_n \bx_n\| = |b_1|+ \sum_{n=1}^\infty |b_{n+1} - b_n| .
\]
In particular, this basis is normalized with $\|\bx_n\|=2$.

Let $y = \sum_{n=1}^{2N} \bx_n$, so that $\|y\|=1$. Taking
$A=\{2,4, \cdots, 2N\}$ we obtain \[\|S_A(y)\| = \|\sum_{n=1}^{N}
\bx_{2n} \|= 2N\,.\] Thus $k_N \geq  {\|S_A (y)\|}/{\|y\|} \geq
2N.$ By Proposition \ref{Pro4} and Theorem \ref{Pro2}, we then
conclude that $C_N \gtrsim k_N\approx D_N\gtrsim N$.

%

\

\ {\bf Example 5.}
The last example consists of a general procedure showing that $k_N$ and $\mu(N)$ may
essentially be arbitrary.

Let $\SX$ and $\SY$ be Banach spaces with respective (normalized)
bases $\{\be_j\}$ and $\{\ff_j\}$. We consider the direct sum
space $\SX\oplus\SY$, consisting on pairs $(x,y)\in\SX\times\SY$
with norm given by $\|x\|_\SX+\|y\|_\SY$. Clearly, the
system\footnote{As usual, in $\SX\oplus\SY$ one just writes $x$ in
place of $(x,0)$, and $y$ in place of $(0,y)$.}
$\{\be_1,\ff_1,\be_2,\ff_2,\ldots\}$ is a basis of $\SX\oplus\SY$.
Moreover, we the have the following.

\begin{proposition}
\label{x+y} If $\{\be_j\}_{j=1}^\infty$ is quasi-greedy in $\SX$ and
$\{\ff_j\}_{j=1}^\infty$ quasi-greedy in $\SY$, then
$\{\be_j,\ff_j\}_{j=1}^\infty$ is quasi-greedy in $\SX\oplus\SY$.
Moreover, \Benu \item[(a)] $k^{\SX\oplus\SY}_N\,=\,\max\{k_N^\SX,
k_N^\SY\}$ \item[(b)]
$h_r^{\SX\oplus\SY}(N)\,\approx\,\max\{h_r^\SX(N), h_r^\SY(N)\}$
\item[(c)] $\min\{h_\ell^\SX(N/2),
h_\ell^\SY(N/2)\}\,\lesssim\,h_\ell^{\SX\oplus\SY}(N)\,\leq\,\min\{h_\ell^\SX(N),
h_\ell^\SY(N)\}$.\Eenu
\end{proposition}
\begin{proof}
 The proof is elementary. Quasi-greediness follows from
\[
\|G_N(x+y)\|_{\SX\oplus\SY}\,\leq \,\max_{0\leq k\leq N}\,
(\|G_k(x)\|_{\SX}+\|G_{N-k}(y)\|_\SY)\,\lesssim\,\|x\|_\SX+\|y\|_\SY.
\]
The statement (a) is an easy consequence of the identity
\[ k^{\SX\oplus\SY}_N=\,\sup_{|A_1|+|A_2|\leq
N}\sup_{{x\in \SX,y\in\SY}\atop{\|x\|_\SX+\|y\|_\SY\not=0}}
\frac{\|S_{A_1}x\|_\SX+\|S_{A_2}y\|_\SY}{\|x\|_\SX+\|y\|_\SY}.\]
Similarly, (b) follows from\[
h^{\SX\oplus\SY}_r(N)\,=\,\sup_{|A_1|+|A_2|= N} (\|\sum_{i\in
A_1}\be_i\|_\SX+\|\sum_{j\in A_2}\ff_j\|_\SY).
\]
For (c) one uses
\[
h^{\SX\oplus\SY}_\ell(N)\,=\,\inf_{|A_1|+|A_2|= N} (\|\sum_{i\in
A_1}\be_i\|_\SX+\|\sum_{j\in A_2}\ff_j\|_\SY)
\leq\,\min\{h_\ell^\SX(N), h_\ell^\SY(N)\}.
\]
For the lower bound notice that
\[
h^{\SX\oplus\SY}_\ell(N)\,\geq\,\min_{1\leq k\leq
N}\{h_\ell^\SX(k)+
h_\ell^\SY(N-k)\}\,\gtrsim\,\min\{h_\ell^\SX(N/2),
h_\ell^\SY(N/2)\},
\]
where in the last step one splits the cases $k\leq N/2$ and
$k>N/2$, and uses that $h^\SX_\ell$ and $h^\SY_\ell$ are almost
increasing (by quasi-greediness; see \eqref{ap1}).
\end{proof}

As a particular case, consider $\SX$ as in Example 1, so that \Be
k_N^\SX\approx \log N\mand h_\ell^\SX(N)\approx h_r^\SX(N)\approx
N.\label{6x}\Ee Consider also the space $\SY$ given by the closure
of $c_{00}$ with the norm
 \Be \|(y_j)\|_\SY=\sup_A
\sum_{j \in A} |y_j|/(1+\log |A|)^\al<\infty,\label{6y}\Ee where
$\al>0$ is fixed. One easily checks that
\[k_N^\SY=1\mand h_\ell^\SY(N)=
h_r^\SY(N)= \frac N{(1+\log N)^\al}.\] Combining \eqref{6x},
\eqref{6y} and Proposition \ref{x+y} we see that $\SX\oplus\SY$
has \Be k_N^{\SX\oplus\SY}\approx \log N \mand
\mu^{\SX\oplus\SY}(N)\approx(\log N)^\al.\label{6kN}\Ee

To show possible applications of our results, construct a quasi-greedy
basis with $k_N \approx \mu(N) \approx \log N$. Theorem \ref{Th1}
shows that $C_N \lesssim \log N$. This is an improvement over
previously known estimates: both \cite[Theorem 2.1]{TYY2011b} and
\cite[Thm 1.1]{H} only yield $C_N \lesssim (\log N)^2$.

\section{Limitations}\label{sec:limits}

One could use Theorem \ref{Th3} to show that a given basis is not
quasi-greedy, by establishing that its $k_N$  constants grow
faster than $c\,\log N$ for any $c>0$. We also know that
$k_N=O(1)$ characterizes unconditional bases. It is then fair to
ask whether the slow growth  $k_N \leq c\log N$ could characterize
quasi-greedy bases. Below we show that it is not the case.

\begin{proposition}\label{prop:growth}
Suppose a sequence $1 \leq c_1 \leq c_2 \leq \ldots$ increases
without a bound (perhaps very slowly). Then there exists a Banach
space $\SX$ with a normalized basis $(\be_i)$ such that $k_{2N}
\leq c_N$, and $(\be_i)$ is not quasi-greedy.
\end{proposition}

\begin{proof}
Without loss of generality, we may assume $c_n \leq n$. Furthermore,
passing to the sequence
$$
c_j^\prime = \min\big\{c_j, \min_{i<j} c_i j/i \big\}
$$
if necessary, we may assume that the sequence $(c_j/j)$ is
non-increasing.

For $j \in \N$, let $S_j = \{5^j + 1, 5^j + 2, \ldots, 5^j + 2j
\}$. Define a norm on $c_{00}$ by setting, for $x = (x_i)$,
$$
\|x\| = \max \Big\{ \sup_i |x_i| , \sup_n \sup_j
 \frac{c_j}{2j} \big| \sum_{{i \in S_j}\atop{i\leq n}} (-1)^i x_i \big| \Big\} ,
$$
and let $\SX$ be the completion of $c_{00}$ in this norm. Denote
the canonical basis in $\SX$ by $(\be_i)$, which is clearly a
monotone basis.

Note that $(\be_i)$ is not unconditional with constant coefficients,
hence not quasi-greedy. Indeed, for $j \in \N$, let $S_j^\prime =
\{5^j + 1, 5^j + 3, \ldots, 5^j + 2j - 1 \}$. Then
$$
\big\| \sum_{i \in S_j} \be_i \big\| = 1 , \, \, {\mathrm{while}} \,
\, \big\| \sum_{i \in S_j^\prime} \be_i \big\| = \frac{c_j}{2} .
$$

It remains to show that $\|S_B x\| \leq c_N$ whenever $|B| \leq
2N$, and $\|x\| \leq 1$. Write $x = \sum_i x_i \be_i$, with
$\sup_i |x_i| \leq 1$. Let $A_j = S_j \cap B$. Then $\|S_B x\|
\leq \max\{1, C\}$, where
 \Bea
C & = & \sup_n\,\sup_j \frac{c_j}{2j} \big| \sum_{{i \in
A_j}\atop{i\leq n}} (-1)^i x_i \big| \\ & =& \sup_n\,\max \Big\{
\max_{j \leq N} \frac{c_j}{2j} \big| \sum_{{i \in A_j}\atop{i\leq
n}} (-1)^i x_i \big| ,\, \sup_{j > N} \frac{c_j}{2j} \big|
\sum_{{i \in A_j}\atop{i\leq n}} (-1)^i x_i \big| \Big\}
\\ & \leq & \max \Big\{ \max_{j \leq N} c_j ,\, \sup_{j
> N} \frac{Nc_j }{j} \Big\} \,=\, c_N . \Eea
\end{proof}

\section{Proof of Theorem \ref{th2new}}


We must show that \Be \big\|x-G_{N+k}(x)\big\|\,\leq\,c\,\Big(1+
\frac{h_r(N)}{h_l(k)}\Big)\,\sigma_N(x),\quad \forall\;x\in\SX.
\label{S1}\Ee
Observe that this quantifies how many iterations
of the greedy algorithm may be necessary to reach $\sigma_N(x)$.
As mentioned in $\S1$ estimates of this sort were obtained in
\cite{KT2,DKKT,Wo3,KaT}, with its roots going back to the work
 of Konyagin and Temlyakov \cite{KT}. Our proof is a suitable combination
 of these ideas, plus the argument we used in Theorem \ref{Th1}
to control the term $\|S_{\GA}(x)\|$.

\medskip

More precisely, take any $p\in\Sigma_N$, say with $\supp p\subset
P$ and $|P|=N$. We shall compare $G_{N+k}(x)$ with
$p+G_k(x-p)\in\Sigma_{N+k}$. Let $\Ga=\supp G_{N+k}(x)$ and notice
that $B=\supp G_k(x-p)$ can be chosen\footnote{Different choices
may appear in case of ties in the size of coefficients.} such that
$B\setminus P\subset\Ga$.


\

 Then \Beas
\big\|x-G_{N+k}(x)\big\| & = & \|x-S_\Ga x\|\\
& \leq & \|x-S_{P\cup B}(x)\|+\|S_{P\cup B}(x)-S_\Ga x\|\\
& \leq & \|x-S_{P\cup B}(x)\| + \|S_{(P\cup
B)\setminus\Ga}(x)\|+\|S_{\Ga\setminus(P\cup
B)}(x)\|=I_1+I_2+I_3.\Eeas The third term can be written as
\[
I_3=\|G_{|\Ga\setminus(P\cup B)|}(x-S_{P\cup
B}x)\|\leq\,K\,\|x-S_{P\cup B}(x)\|,
\]
so it suffices to estimate the first two terms.
\medskip

We begin with $I_2$. Since $B\setminus P\subset\Ga$, we have
$(P\cup B)\setminus\Ga=P\setminus\Ga$.
%
 Use \eqref{dil} and the definition of $h_r$
to obtain
\[ I_2 \, = \,
\|S_\PG(x)\|\,\leq c_2\,\max_\PG|a_k(x)|\,\|\sum_{\PG} \be_k \| \leq
c_2\,\max_\PG|a_k(x)|\,h_r(|\PG|).\] Now using Lemma \ref{L1}\Beas
\max_\PG|a_k(x)|& \leq & \min_\GP |a_k(x)|\, =\,\min_\GP
\big|a_k\big(S_\GP(x-p)\big)\big|\\
& \leq & a^*_{|\GP|}(x-p)\,=\min_{\GP} |a_k (x-p)|\,.\Eeas Thus,
by \eqref{dil} again
\[ \min_\GP|a_k(x-p)| \leq  \frac{1}{c_1}\frac{\|S_{\GP} (x-p)\|}
{\| \sum_{\GP} \be_k\|}\, = \frac{1}{c_1}\frac{\|G_{|\GP|}
(x-p)\|} {\| \sum_{\GP} \be_k\|}\, \leq
\frac{K}{c_1}\,\frac{\|x-p\|}{h_l(|\GP|)}\,.\] Combining these
inequalities we obtain
 \Beas I_2 &\leq & \frac{c_2}{c_1}\, K\,
\frac{h_r(|\PG|)}{h_l(|\GP|)}\,\|x-p\|. \Eeas Observe that since
the basis is quasi-greedy, if $A \subset B$ we have $\|\bone_A \|
\leq K \|\bone_B\|\,.$  Hence, $h_r(|\PG\|) \leq K h_r(N)$ since
$|\PG| \leq N\,.$ Similarly, $h_l(|\GP|) \geq \frac{1}{K} h_l(k)$
since $k \leq |\GP|\,.$  Thus, $$I_2 \leq
\frac{c_2}{c_1}\,K^3\,\frac{h_r(N)}{h_l(k)}\,\|x-p\|\,.$$

\medskip

We now estimate $I_1$, following the approach in \cite{DKKT};
namely, \Beas I_1 & = & \|x-S_{P\cup B}(x)\| \, = \,
\|x-p\,-\,S_{P\cup
B}(x-p)\|\\
& \leq & \|x-p\,-\,S_{B}(x-p)\| +
\|S_{\PB}(x-p)\|\,=\,J_{1}+J_{2}.\Eeas Clearly\[ J_{1} \,=\,
\|x-p\,-\,G_k(x-p)\|\,\leq\,(1+K)\,\|x-p\|.\] To estimate $J_2$
use \eqref{dil} and the quasi-greediness of the basis to obtain
\Beas J_{2} & \leq & c_2
\,\max_{\PB}\big|a_j(x-p)\big|\,h_r(|\PB|)\, \leq c_2 \,
\min_{B}\big|a_j(x-p)\big|\,h_r(|\PB|)\\ & \leq & \frac{c_2}{c_1}
\frac{h_r(|\PB|)}{h_l(|B|)}\,\big\|G_{|B|}(x-p)\big\|
\,\leq\,\frac{c_2}{c_1}\,K\,\frac{h_r(|\PB|)}{h_l(k)}\,\|x-p\|.
\Eeas  As before,
$h_r(|\PB|) \leq K\, h_r(N)$, so we deduce $$ J_2 \leq
\,\frac{c_2}{c_1}\,K^2\,\frac{h_r(N)}{h_l(k)}\,\|x-p\|.$$ Thus,
putting together all cases we obtain
\[
\big\|x-G_{N+k}(x)\big\|\,\leq\,c\,\Big(1+
\frac{h_r(N)}{h_l(k)}\Big)\,\|x-p\|,\quad \forall\;p\in\Sigma_N,
\]
with the constant $c$ of the order $K^3c_2/c_1 =O(K^6)$.

\medskip

\sline\emph{Remarks:}

 \Bi

\item[$\bullet$] As pointed out in \cite{KaT}, \eqref{S1} improves
over \eqref{1} in some situations. For instance, assume
$h_l(N)=N^\al$ and $h_r(N)=N^\beta$ with $0<\al<\beta\leq1$. If $x$
is such that $\sigma_N(x)=O(N^{-r})$ then \eqref{1} gives
\[
\big\|x-G_M(x)\big\|\,\lesssim\,M^{\beta-\al}\,\sigma_M(x)\,\lesssim\,M^{-[r-(\beta-\al)]}
\]
while \eqref{S1} gives, when $M=k+N$ with $k\approx M\approx
N^{\beta/\al}$,
\[
\big\|x-G_M(x)\big\|\,\lesssim\,\sigma_{cM^{\al/\beta}}(x)\,\lesssim\,M^{-r\al/\beta}.
\]
When $r<\beta$, the second estimate improves over the first (for
large $M$). In the language of approximation spaces (see e.g. \cite{GHN}), these estimates
can also be read as
\[
A^r_\infty(\SX)\,\hookrightarrow\,
\mathscr{G}_{\infty}^{\max\{\frac{r\alpha}\beta,\,r-(\beta-\al)\}}(\SX).
\]

\item[$\bullet$] The estimate \eqref{S1} is only interesting when
$\lim_{k\to\infty}h_l(k)=\infty$ (so that $h_l(k)$ can reach
$h_r(N)$), and cannot be improved when $h_l$ is just bounded and
$h_r(N)\to\infty$. To see the latter, arguing as in Lemma \ref{3.2} one can find disjoint sets $A,B$ with $|A|=N+k$, $|B|=N$ and
$\|\bone_B\|/\|\bone_A\|\gtrsim h_r(N)$. Setting
$x=2\bone_A+\bone_B$ one sees that $\|x-G_{N+k}(x)\|/\sigma_N(x)\geq
\|\bone_B\|/\|\bone_A\|\gtrsim h_r(N)\to\infty$.

 \Ei

\section{Some questions} \label{questions}

Quasi-greedy bases  in $L^p(\mathbb T^d)$ were studied in
\cite{TYY2011a,TYY2011b}. In these cases one always has $\mu(N)
\lesssim N^{|\frac{1}{p} - \frac{1}{2}|}$, from the type and
cotype properties of $L^p(\mathbb T^d)$, $1 < p < \infty$. Hence,
using Theorem \ref{Th1} (and \eqref{kNlogN}) one obtains that
$C_N\lesssim N^{|\frac{1}{p} - \frac{1}{2}|}\,,$ when $p\not=2$, a
result which was proved in \cite{TYY2011b}.  When $p=2$, this
argument only gives $C_N \lesssim \log N$, a result which goes
back to \cite{Wo}.

\bline {\sc Question 1.} (Asked in \cite{TYY2011a,DST}).
\emph{Investigate whether, for quasi-greedy bases in a Hilbert
space, the inequality $C_N\lesssim \log N$ can be replaced by a
slower growing factor.}

\bline Recently, P. Wojtaszczyk \cite{Wo4} has showed us that, for
quasi-greedy bases in $L^2$, say with constant $K$, there exists
$\al=\al(K)<1$, such that $C_N \lesssim (\log N)^\alpha\,$. Also,
it can be deduced from the results in \cite{DST} that $C_N\lesssim
(\log N)^{1/2}$ for all quasi-greedy \emph{besselian}\footnote{Here besselian means
$\|(a_k)\|_{\ell^2}\lesssim \|\sum_k a_k\be_k\|_2$, for all
finitely scalars $(a_k)$.} bases in $L^2$. However, no examples
where these bounds are attained seem to be known.

\bline Consider now the trigonometric system $\mathcal T^d =
\{e^{ikx} : k\in \mathbb Z^d\}$ in $L^p(\mathbb T^d)$, $ 1\!
\leq\! p \!\leq \!\infty$ (understood as $C(\mathbb T^d)$ for
$p=\infty$). Notice that $\mathcal{T}^d$ is not quasi-greedy in
$L^p$, $p\not=2$. It was proved in \cite[Theorem 2.1]{Tem1998}
that one also has
$$
  C_N = C_N(\mathcal T^d,L^p(\mathbb T^d)) \,\lesssim \,N^{|\frac{1}{p} -
   \frac{1}{2}|}\,, \quad 1\leq p \leq \infty\,.
$$

\vspace{.3cm}

{\sc Question 2.} (Asked by V. N. Temlyakov at the
\textit{Concentration week on greedy algorithms in Banach spaces
and compressed sensing} held on July 18-22, 2011 at Texas A\&M
University.)

a) Characterize those systems $\mathcal B$ in $L^p(\mathbb T^d)$,
$1\leq p \leq \infty\,,$ such that \[C_N(\mathcal B,L^p(\mathbb
T^d)) \lesssim N^{|\frac{1}{p} - \frac{1}{2}|}\,,\quad
N=1,2,\ldots\] Notice that if $1 < p \neq 2  < \infty\,,$ the
characterization must be satisfied by $\mathcal T^d$ as well as
any quasi-greedy basis.

More generally, let $v(N)$ be an increasing function of $N$.

b) Characterize, in a Banach space $\SX$, those systems $\mathcal
B$  (not necessarily quasi-greedy) for which $ C_N(\mathcal B,
\SX) \lesssim v(N)\,.$

\medskip

\section{Appendix: proof of \eqref{dil}} \label{appendix}

 The proof suggested in \cite{DKKT} for the inequalities in \eqref{dil}
 is only valid for \emph{real} scalars $a_k\in\SR$;
 we give below a minor modification of their argument that
 establishes \eqref{dil} also for \emph{complex} scalars $a_k$.
 Below $K$ denotes the quasi-greedy constant in $\SX$.

The first two lemmas are similar to \cite[Prop 2]{Wo}.
 \begin{lemma}
 \label{lap.1} Let $\{\be_j\}_{j=1}^\infty$ is a  quasi-greedy basis in a
Banach space $\SX$.
 For all $\beta_j\in\SC$ with $|\beta_j|=1$, and
 all finite sets $A_1\subset A$, it holds\Be
 \big\|\sum_{j\in A_1}\beta_j\be_j\big\|\leq K\,\big\|
 \sum_{j\in A}\beta_j\be_j\big\|. \label{ap1}\Ee
 \end{lemma}
\begin{proof}
Call $A_2=A\setminus A_1$. For $\e>0$, define $x=\sum_{j\in
A_1}(1+\e)\beta_j\be_j+\sum_{j\in A_2}\beta_j\be_j$. Then \[
\|G_{|A_1|}(x)\|=(1+\e)\big\|\sum_{j\in A_1}\beta_j\be_j\big\|\leq
K\|x\|\,=\,\big\|(1+\e)\sum_{j\in A_1}\beta_j\be_j+
 \sum_{j\in A_2}\beta_j\be_j\big\|.\]
 Letting  $\e\to0$ we obtain \eqref{ap1}.
\end{proof}

\begin{lemma}
 \label{lap.2} Let $\{\be_j\}_{j=1}^\infty$ is a  quasi-greedy basis in a
Banach space $\SX$.
 For all $\e_j\in\{\pm1,\pm i\}$, and all finite
sets $A$ it holds\Be \tfrac1{4K}\,\big\|\sum_{j\in A}\be_j\big\|\leq
\big\|\sum_{j\in A}\e_j\be_j\big\|\leq \,4K\,\big\|\sum_{j\in
A}\be_j\big\|. \label{ap2}\Ee
 \end{lemma}
\begin{proof}
Call  $A_k=\{j\in A\mid \e_j=i^k\}$, $k=1,\ldots,4$. Then, the
triangle inequality and \eqref{ap1} (with all $\beta_j=1$) give \[
\big\|\sum_{j\in A}\e_j\be_j\big\|\leq
\,\sum_{k=1}^4\,\big\|\sum_{j\in A_k}\be_j\big\|\,\leq\,
4K\,\big\|\sum_{j\in A}\be_j\big\|,\] establishing the right hand
side of \eqref{ap2}. Arguing similarly,
\[
\big\|\sum_{j\in A}\be_j\big\|\leq \,\sum_{k=1}^4\,\big\|\sum_{j\in
A_k}\be_j\big\|\,\leq\, 4K\,\big\|\sum_{j\in A}\e_j\be_j\big\|
\]
where we have now used \eqref{ap1} with $\beta_j=\e_j$.
\end{proof}

\begin{lemma}
 \label{lap.3}
 For all complex $\beta=a+ib$ with $|a|+|b|\leq 1$, and for all $x,y\in\SX$ it holds
 \Be  \big\|x+\beta y\big\|\leq
\,\max\big\{\|x\pm y\|,\|x\pm i y\|\big\}. \label{ap3}\Ee
 \end{lemma}
\begin{proof}
We may assume that $a\in[0,1)$. Then \Bea \|x+\beta y\| & \leq &
\|ax + ay\|+\|(1-a)x+iby\|\nonumber\\ & = & a\|x+y\|+(1-a)\|x+i\ga
y\|, \label{apaux1}\Eea where we have set $\ga=b/(1-a)$, which is
a real number with $|\ga|\leq1$. Now \Beas\|x+i\ga y\| & = &
\big\|\tfrac{1-\ga}2(x-iy)\,+\,\tfrac{1+\ga}2(x+iy)\big\|
\\
& \leq
&\tfrac{1-\ga}2\big\|x-iy\big\|\,+\,\tfrac{1+\ga}2\big\|x+iy\big\|\,\leq\,\max\,\big\|x\pm
iy\big\|,\Eeas where we have used that $-1\leq\ga\leq1$. Inserting
this into \eqref{apaux1} easily leads to \eqref{ap3}.
\end{proof}

We now justify the right hand bound in \eqref{dil}. For a complex
number $\al=a+ib$ we shall denote $|\al|_1=|a|+|b|$. Then,
iterating the previous lemma we obtain \Bea \big\|\sum_{j\in
A}\al_j\be_j\big\| & \leq & \max_{j\in
A}|\al_j|_1\,\max_{\e_j\in\{\pm1,\pm
i\}}\,\big\|\sum_{j\in A}\e_j\be_j\big\|\nonumber\\
& \leq & 4\,\sqrt 2\,K\,\max_{j\in A}|\al_j|\,\big\|\sum_{j\in
A}\be_j\big\|,\label{dilr}\Eea where in the last step we have used
Lemma \ref{lap.2} and the trivial estimate $|\al|_1\leq \sqrt2
|\al|$.

We can now state a slightly more general version of Lemma
\ref{lap.2}.

\begin{lemma}
 \label{lap.4} Let $\{\be_j\}_{j=1}^\infty$ is a  quasi-greedy basis in a
Banach space $\SX$.
 For all $\e_j\in\SC$ with $|\e_j|=1$, and all finite
sets $A$ it holds\Be \tfrac1{4\sqrt2\,K}\,\big\|\sum_{j\in
A}\be_j\big\|\leq \big\|\sum_{j\in A}\e_j\be_j\big\|\leq
\,4\sqrt2\,K\,\big\|\sum_{j\in A}\be_j\big\|. \label{ap4}\Ee
 \end{lemma}
\begin{proof}
The right hand side is a special case of \eqref{dilr}. To obtain the
left hand side, we consider the system $\{\tbe_j:=\e_j\be_j\}$,
which is also a quasi-greedy basis in $\SX$ with the same constant
$K$. Thus, \eqref{dilr} for this system (with $\al_j=\bar{\e}_j$)
gives
\[
\big\|\sum_{j\in A}\bar{\e}_j\tbe_j\big\|\leq
\,4\sqrt2\,K\,\big\|\sum_{j\in A}\tbe_j\big\|,
\]
but this is the same as the left hand side of \eqref{ap4}.
\end{proof}

We turn now to the left hand inequality in \eqref{dil}, for which we
follow the arguments in \cite[p. 579]{DKKT}. We shall prove that, if
$A$ is finite then\Be \big\|\sum_{j\in A}\alpha_j\be_j\big\|\geq
\,\tfrac1{8\sqrt2\,K^2}\,\min_{j\in A}|\al_j|\,\big\|\sum_{j\in
A}\be_j\big\|.\label{dill} \Ee Write each scalar
$\al_j=\e_j|\al_j|$, with $\e_j\in\SC$ such that $|\e_j|=1$, and
consider a permutation $\{j_1,\ldots,j_N\}$ of $A$ such that
$|\al_{j_1}|\geq |\al_{j_2}|\geq\ldots\geq|\al_{j_N}|$. Let
$x=\sum_{j\in A}\alpha_j\be_j$ and set $G_0(x)=0$. Then \Bea
 |\al_{j_N}|\,\big\|\sum_{\ell=1}^N\e_{j_\ell}\be_{j_\ell}\big\| &
 = &
 |\al_{j_N}|\,\Big\|\sum_{\ell=1}^N\tfrac1{|\al_{j_\ell}|}\,\big(G_\ell(x)-G_{\ell-1}(x)\big)\Big\|\label{apaux2}\\
&
 = &
 |\al_{j_N}|\,\Big\|\sum_{\ell=1}^{N-1}\Big(\tfrac1{|\al_{j_\ell}|}-\tfrac1{|\al_{j_{\ell+1}}|}\Big)G_\ell(x)+
 \tfrac1{|\al_{j_N}|}G_{N}(x)\Big\|\nonumber\\
&
 \leq &
 |\al_{j_N}|\,\Big[\tfrac1{|\al_{j_N}|}\,+\,
 \sum_{\ell=1}^{N-1}\big(\tfrac1{|\al_{j_{\ell+1}}|}-\tfrac1{|\al_{j_{\ell}}|}\big)\Big]\,K\,\big\|x\big\|\,
 \leq\, 2K\,\big\|x\big\|.\nonumber\Eea
On the other hand, by Lemma \ref{lap.4}, the expression on the left
of \eqref{apaux2} can be estimated from below by
$|\al_{j_N}|\,\|\sum_{j\in A}\be_j\|/{4\sqrt2\,K}$, from which
\eqref{dill} follows.

Thus, putting together \eqref{dilr} and \eqref{dill} we have shown

\begin{proposition}
\label{p1.ap} Let $\{\be_j\}_{j=1}^\infty$ is a  quasi-greedy basis
in a Banach space $\SX$. If $A$ is finite and $\al_j\in\SC$ then\[
\tfrac1{8\sqrt2\,K^2}\,\min_{j\in A}|\al_j|\,\big\|\sum_{j\in
A}\be_j\big\|\,\leq\,\big\|\sum_{j\in
A}\alpha_j\be_j\big\|\,\leq\,4\,\sqrt 2\,K\,\max_{j\in
A}|\al_j|\,\big\|\sum_{j\in A}\be_j\big\|.
\]\end{proposition}

\bibliographystyle{plain}

\begin{thebibliography}{1}


\bibitem{bechler}
\textsc{P. Bechler},
 \emph{Inequivalence of wavelet systems in $L_1(\Bbb R^d)$and ${\rm BV}(\Bbb R^d)$}.
 Bull. Pol. Acad. Sci. Math. 53 (2005), no. 1, 25--37.


\bibitem{CDDD}
\textsc{A. Cohen, W. Dahmen, I. Daubechies, R. DeVore},
\emph{Harmonic analysis of the space $BV$}. Rev. Mat. Iberoamericana
19 (2003), no. 1, 235--263.


\bibitem{CDPX}
\textsc{A. Cohen, R. DeVore, P. Petrushev, H. Xu}, \emph{Nonlinear
approximation and the space ${\rm BV}({\bf R}^2)$}. Amer. J. Math.
121 (1999), no. 3, 587--628.

\bibitem{DKK}
\textsc{S.J. Dilworth, N.J. Kalton, D. Kutzarova}, \emph{On the existence of almost greedy bases in Banach spaces},  Studia Math. 159 (2003), no. 1, 67--101.

\bibitem{DST}
\textsc{S.J. Dilworth, M. Soto-Bajo, and V.N.
Temlyakov}, \emph{Quasi-greedy bases and Lebesgue-type
inequalities}, preprint 2012.

\bibitem{DKKT}
\textsc{S.J. Dilworth, N.J. Kalton, D. Kutzarova, and V.N.
Temlyakov}, \emph{The Thresholding Greedy Algorithm, Greedy Bases,
and Duality}, Constr. Approx., 19, (2003),575--597.

\bibitem{DM}
\textsc{S.J. Dilworth, D. Mitra}, ``A conditional quasi-greedy basis
of $\ell^1$'', Studia Math. 144 (2001), 95-100.

\bibitem{GHM2008}
\textsc{G. Garrig\'os, E. Hern\'andez, J.M. Martell},
\emph{Wavelets, {O}rlicz spaces and greedy bases},  Appl. Compt.
Harmon. Anal., (24): (2008), 70--93.

\bibitem{GHN}
\textsc{G. Garrig\'os, E. Hern\'andez, M. de Natividade},
\emph{Democracy functions and optimal embeddings for approximation
spaces}, Advances in Computational Mathematics, Accepted, 2011.
(Online: 23 September 2011), pp. 1-29, doi:10.1007/s10444-011-9197-0

\bibitem{H}
     \textsc{E. Hern\'andez},
     \emph{Lebesgue-type inequalities for quasi-greedy bases}.
     Preprint 2011. ArXiv: 1111.0460v2 [matFA] 16 Nov 2011.

\bibitem{KaT}
\textsc{A. Kamont and V.N. Temlyakov}, \emph{Greedy approximation
and the multivariate Haar system}, Studia Math,  161 (3), (2004),
199--223.

\bibitem{KT}
\textsc{S.V. Konyagin and V.N. Temlyakov}, \emph{A remark on greedy
approximation in Banach spaces}, East. J. Approx. 5, (1999),
365--379.

\bibitem{KT2}
\textsc{S.V. Konyagin and V.N. Temlyakov},
``Greedy approximation with regard to bases and general minimal systems''.
Serdica Math. J. 28 (2002), no. 4, 305-328,


\bibitem{L}
\textsc{J. Lindenstrauss}, ``On a certain subspace of $\ell^1$''.
Bull. Acad. Polon. Sci. {\bf 12} (1964), 539-542.

\bibitem{LZ}
\textsc{J. Lindenstrauss, L. Tzafriri}, \emph{Classical Banach
spaces}, vol I, Springer–Verlag 1977.

\bibitem{Sin}
\textsc{I. Singer}, \emph{Bases in Banach Spaces}, vol. I,
Springer–Verlag 1970.


\bibitem{Tem1998}
\textsc{V. N. Temlyakov}, \emph{Greedy algorithm and n- term
trigonometric approximation}, Const.Approx., 14, (1998), 569--587.

\bibitem{T1}
\textsc{V.N. Temlyakov}, \emph{Greedy approximation}, Acta Numerica
(2008), 335--409.

\bibitem{T2}
\textsc{V.N. Temlyakov}, \emph{Greedy approximation}, Cambridge
University Press, 2011.

\bibitem{TYY2011a}
\textsc{V. N. Temlyakov, M. Yang, P. Ye}, \emph{Greedy approximation
with regard to non-greedy bases}, Adv. in Comp. Math., 34, (2011),
319--337.

\bibitem{TYY2011b}
\textsc{V. N. Temlyakov, M. Yang, P. Ye}, \emph{Lebesgue-type
inequalities for greedy approximation with respect to quasi-greedy
bases}, East J. Approx {\bf 17} (2011), 127--138.

\bibitem{Wo} \textsc{P. Wojtaszczyk},
\emph{Greedy Algorithm for General Biorthogonal Systems}, Journal of
Approximation Theory, 107, (2000), 293--314.

\bibitem{Wo3} \textsc{P. Wojtaszczyk},
\emph{Greedy type bases in Banach spaces}.
Constructive theory of functions, 136--155, DARBA, Sofia, 2003.

\bibitem{Wo2} \textsc{P. Wojtaszczyk},
\emph{Projections and non-linear approximation in the space ${\rm
BV}(\Bbb R^d)$}. Proc. London Math. Soc. 87 (2003), no. 2, 471--497.

\bibitem{Wo4} \textsc{P. Wojtaszczyk}, personal communication, June 2012.


\end{thebibliography}

\vskip 1truemm

\end{document}